\documentclass{amsart}
\usepackage{latexsym}
\usepackage{amsmath}
\usepackage{amsfonts}
\usepackage{amsthm}
\usepackage{amssymb}
\usepackage{amsrefs}
\usepackage[utf8]{inputenc}
\usepackage[T1]{fontenc}
\usepackage{xcolor}

\title[Red-and-black game]{A new class of probabilities in the $N$-person red-and-black game}
\author{W\l{}odzimierz Fechner}
\author{Maria S\l{}omian}
\address{Institute of Mathematics, Lodz University of Technology, al. Politechniki 8, 93-590 \L\'od\'z, Poland}
\email{wlodzimierz.fechner@p.lodz.pl}
\email{marysia.slomian@gmail.com}

\newtheorem*{thm}{Theorem}
\newtheorem{cor}{Corollary}
\newtheorem{prop}{Proposition}
\theoremstyle{remark}
\newtheorem{rem}{Remark}
\newtheorem{ex}{Example}
\theoremstyle{definition}
\newtheorem*{df}{Definition}

\newcommand{\R}{\mathbb{R}}

\newcommand{\f}{\varphi}

\renewcommand{\(}{\left(} \renewcommand{\)}{\right)}


\begin{document}

\begin{abstract}
We discuss a model of a $N$-person, non-cooperative stochastic game, inspired by the discrete version of the red-and-black gambling problem introduced by Dubins and Savage in 1965. Our main theorem generalizes a result of Pontiggia from 2007 which provides conditions upon which bold strategies for all players form a Nash equilibrium. Our tool is a functional inequality introduced and discussed in the present paper. It allows us to avoid rather restrictive assumptions of super-multiplicativity and super-additivity, which appear in Pontiggia's and other authors' works. We terminate the paper with some examples which in particular show that our approach leads to a larger class of probability functions than existed in the literature so far.
\end{abstract}

\maketitle 

\section{Introduction}

The red-and-black game problem goes back to 1965 when Lester E. Dubins and Leonard J. Savage wrote their seminal book \textit{How to gamble if you must} \cite{DS}. They presented a problem, in which a player wanted to gain a certain amount of money, and any other lower prize was not sufficient for their purpose. Having some initial capital and a goal, the authors answer how to play - not whether to play at all - to increase the chances of winning. 

In 2005 L. Pontiggia in \cite{P1} presented two-person game models, where both players aim to get all of the opponent's money. Both models were based on winning probabilities - in the first one, they were proportional to the player's bets, in the second one, they were estimated with additional weight parameter $\omega$. In 2006 M. Chen and S. Hsiau wrote paper \cite{Ch1}, where they introduced the winning probability function $f$ as a function of one variable. They got two results concerning the best strategy for a player, when their opponent's strategy was known and when $f$ has certain properties. What is more, Chen and Hsiau gave a counterexample for Pontiggia's assumption for the $N$-person model, $N \geq 2$ which was shown in \cite{P1}. A year later, in 2007 L. Pontiggia in \cite{P2} presented a new model for $N$-person game, including a gambling house, which on every step of the game had a positive probability of winning all the players' bets. In her model optimal strategy for all participants is to play boldly. 

In this paper, we will explore further the last discussed game, namely the $N$-person game, where the winning sum is non-constant. Based also on \cite{W}, where a generalization for results of \cite{Ch1} has been given, we will show a less restrictive model. Then, we present examples of the functions that meet the assumptions of our main theorem and thus we show that we extend the main result of \cite{P2}.

In our model, we assume that $N$ players have some initial fortunes equal to positive integers and they simultaneously bet some integer parts of their capitals. With some probability depending upon all their bets either one of the players wins the whole pool, or it is taken over by the casino. We consider bold strategy, i.e. the strategy when a player bets all his fortune and in our main result we provide
conditions upon which the profile of bold strategies for all the players form a Nash equilibrium.
 

\section{Rules of the game}

Let $N\geq 2$ be a fixed integer, which denotes the number of players in the game. Next, assume that $(x_1^0, \ldots x_N^0)$ is a vector of positive integers, where $x_j^0$ denotes the initial fortune of $j$-th player. Put $M := \sum_{i=1}^N x_i^0$ (the total amount of money at the beginning of the game) and let $G$ be a positive integer equal to a fixed goal the players aim to reach. We assume that the goal is the same for all players, only one player can win and at least some players have a chance to win. Therefore, we impose the following double inequality:
\begin{equation}
G \leq M < 2G.
\label{MG}
\end{equation}
Denote $S:=\{0, 1, \ldots , M\}$.
We define the state space for the game as
$$ P := \{(x_1, \ldots, x_N):   x_j \in S, j=1, \ldots, N , \, \sum_{i=1}^N x_i \leq M\}.$$
The absorbing states of the game consist of all vectors with one of the coordinates greater or equal to $G$ (when one of the players wins) and also of all vectors for which $\sum_{i=1}^N x_i<G$ (when it is no longer possible to win by any of the players). 

Now, we define an action set of Player $j$ when current fortunes of the players are equal to $(x_1, \ldots, x_j, \ldots , x_N)$:
$$A_j(x_1, \ldots, x_N) := \begin{cases} 
       \{1, \ldots, x_j\}, & \text{if } x_j \in \{1, \ldots, G-1\}, \\
      \{0\}, & \text{if } x_j \in \{0, G\},
   \end{cases}$$
and his payoff function:
\begin{equation}
W_j(x_1, \ldots, x_N) := \begin{cases} 
       1, & \text{if } x_j \geq G, \\
      0, & \text{if } x_j < G. 
   \end{cases}
\label{payoff}
\end{equation}

Note that the game is non-cooperative and each player has no knowledge of the actions simultaneously taken by the others.

\medskip

Now, assume that we are given a function $\Phi\colon S \times P \to [0, 1]$, which represents the probability of winning for the players. More precisely, if bets of the players are equal to $(a_1, \ldots, a_N)$, then the number  $\Phi(a_j; a_1, \ldots, a_N)$ is the probability of victory of Player $j$ (whose bet is $a_j$). 
A special case, 
\begin{equation}
\Phi(a_j; a_1, \dots , a_N) = f\( \frac{a_j}{a_1+ \dots + a_N}\),
\label{fu}
\end{equation}
 where $f$ is super-additive and super-multiplicative, was studied in \cite{P2}. 
As it is stated in \cite{P2}*{Remark 3.1}, the idea of introducing function $f$ is to penalize the players by reducing their probability of winning (for example by the government or by the owner of the casino). Therefore, our model allows us to punish the players more flexibly, not necessarily depending only upon the quotient of the bid of $j$-th player and the sum of all the bids made during this stage of the game.

We impose another assumption upon $\Phi$, which is in particular fulfilled by all mappings of the form \eqref{fu}. Namely, we will assume that
\begin{equation}\label{ass}
\textrm{If } \sum_{i=1}^n a_i = \sum_{i=1}^n b_i \textrm{ and } a_j=b_j, \textrm{ then } \Phi(a_j; a_1, \dots , a_N) = \Phi(b_j; b_1, \dots , b_N).
\end{equation}

The next condition guarantees that the total probability does not exceed one:
\begin{equation}\label{1}
    \sum_{i=1}^N \Phi(a_i; a_1, \ldots, a_N) \leq 1.
\end{equation}
For technical reasons we will need another natural condition that the probability of victory of a given player is equal to zero if his bet is equal to zero (which, according to the rules of the game, is possible only if he lost all his capital at an earlier stage of the game):
\begin{equation}
\label{00}
\Phi(a_j; a_1, \ldots, a_N) =0 \quad \textrm{if} \quad a_j = 0.
\end{equation}

Now, we are at the point of defining precisely the game's law of motion. Let us fix a positive integer $m$, a current stage of the game. Let $X_{m,j}$ be a random variable that is equal to the fortune of $j$-th player at time $m$. By $a_{m,j}$ we denote an amount which he bids at this stage of the game. By the casino rules $1 \leq a_{m,j} \leq X_{m,j}$ and $X_{1,j}<G$ for $j= 1, 2, \ldots, N$. The law of motion is:
\begin{equation}\label{r1}
     X_{1,1} = x_1^0, \qquad X_{1,2} = x_2^0, \qquad \ldots \qquad X_{1,N} = x_N^0  
\end{equation}
\begin{multline}\label{r2}
    (X_{m+1, 1}, \ldots, X_{m+1, N}) =\\ \begin{cases} 
       (X_{m,1} - a_{m, 1}, \ldots, X_{m,N} - a_{m, N}) , \qquad \text{w.p. } 1 - \sum_{i=1}^N \Phi(a_i; a_1, \ldots, a_{N}), \\
       (X_{m,1} - a_{m, 1}, \ldots, X_{m,j} + \sum_{i\neq j} a_{m, i}, \ldots, X_{m,N} - a_{m, N}) , \qquad \text{w.p. } \Phi(a_j; a_1, \ldots, a_{N})
   \end{cases}
\end{multline}
(here ''w.p.'' is an abbreviation of ''with probability'').

Note that inequality \eqref{1} together with the rules \eqref{r1}, \eqref{r2} implies that
\begin{equation}\label{r3}
    \sum_{i=1}^N X_{m+1, i} \leq \sum_{i=1}^N X_{m, i}.  
\end{equation}
Sometimes we will omit double subscripts when it is clear which stage of the game is considered.

\section{The main result}

Our main result is a generalization of a theorem of L. Pontiggia \cite{P2}*{Theorem 3.1}. We replace assumptions of super-additivity and super-multiplicativity with a less restrictive functional inequality, which will be then discussed in the next section.

According to the definition of payoff function \eqref{payoff}, 
the players want to choose a strategy that maximizes the probability of reaching the goal $G$. We provide conditions upon which bold strategies form a Nash equilibrium.

\begin{df}
A strategy of Player $j$ is called \emph{bold} if for every $m = 1, 2, \dots $ one has $ a_{m, j} = X_{m, j}$. A strategy of Player $j$ is called \emph{timid} if for every $m = 1, 2, \dots $ such that $X_{m, j}\geq 1$  one has $ a_{m, j} = 1$.
\end{df}

\begin{thm}\label{t1}
We consider an $N$-person red-and-black game with $N \geq 2$ and with the law of motion described by formulas \eqref{r1} and \eqref{r2}, with probability function $\Phi\colon S\times P \to [0, 1]$ satisfying conditions \eqref{ass} and \eqref{1}.

Assume that for every $j \in\{1, \dots , N\}$ and every choice of $X_{m,1}, \dots , X_{m, N}\in S$ such that $X_{m,1}+ \dots + X_{m, N}\geq G$ functions $f, g \colon S \to [0,1]$ given by
\begin{align}
f(x) &:= \Phi(x;X_{m, 1}, \ldots, X_{m, j-1}, x , X_{m, j+1}, \ldots, X_{m, N}), \quad x \in S,
\label{f}\\
g(x) &:= \sum_{i \neq j} \Phi(x;X_{m, 1}, \ldots, X_{m, i-1}, x , X_{m, i+1}, \ldots, X_{m, N}), \quad x \in S,
\label{g}
\end{align}
satisfy functional inequality
\begin{equation}
f(x) - f(a) \geq g(a)f(x-a),
\label{main}
\end{equation}
for all $a, x \in S$ such that $a \leq x$. Then, a Nash equilibrium for all players is to play boldly.
\end{thm}

\begin{rem}
Functions $f$ and $g$ spoken of in Theorem \ref{t1} depend on $j$ and moreover can be different for different choices of of $X_{m,1}, \dots , X_{m, N}\in S$. However, the functional inequality itself remains the same. For this reason and for the sake of simplicity we have not introduced any additional subscript for $f$ and $g$.  
\end{rem}

\begin{rem}
A special case of the above theorem is when $N=2$ and with no possibility of winning by the casino. In this situation condition \eqref{1} simplifies, which joined with inequality \eqref{main} applied for functions $f$ and $g$ for both players, after an easy calculation leads to the same functional inequality that appeared in \cite{W} in a bit different situation, where bold strategy was shown to be the best response to the timid strategy (\cite{W}*{Theorem 1 and inequality (7) therein}).
\end{rem}

\noindent
\emph{Proof of the Theorem.}
We will follow the idea of the proof of \cite{P2}*{Theorem 3.1}.

Fix $j \in\{1, \dots , N\}$ and assume that all players play boldly. Denote
\begin{equation}
    Q_j(X_1, \ldots, X_N) = \mathbb{P}[\text{Player }j\text{ reaches } G, \text{when the game starts at } (X_1, \ldots X_N)].
\end{equation}
Law of motion of the game at time $m$ is
\begin{multline}\label{r4}
    (X_{m+1, 1}, \ldots, X_{m+1,j} ,\dots , X_{m+1, N}) = \\ \begin{cases} 
       (0, \ldots, 0) , & \text{w.p. } 1 - \sum_{i=1}^N \Phi(X_{m, i}; X_{m,1}, \ldots, X_{m,N}), \\
       (0, \ldots, 0, \sum_{i=1}^N X_{m, i}, 0, \ldots, 0) , & \text{w.p. } \Phi(X_{m, j}; X_{m,1}, \ldots, X_{m,N}).
   \end{cases}
\end{multline}
One can see that when all players adopt bold strategies, then the game terminates after the first round, i.e. $m=1$. Clearly, since all the players bet their entire fortunes, then either one of them wins and thus reaches his goal, or all players go bankrupt and the casino collects their bets. Moreover, since $X_{1,1}+ \dots + X_{1, N}\geq G$, then obviously 
$$ Q_j(0, \dots , 0, \sum_{i =1}^N X_{1,i}, 0, \ldots, 0)=1$$ (Player $j$ won after first round).
Thus we see that the expected return to Player $j$ equals to
\begin{align}
Q_j(X_{1,1}, \ldots, X_{1,N}) &= \Phi(X_{1, j}; X_{1, 1}, \ldots, X_{1, N})\cdot Q_j(0, \ldots, \sum_{i=1}^N X_{1, i}, \ldots, 0)\nonumber  \\&= \Phi(X_{1, j}; X_{1, 1}, \ldots, X_{1, N}). \label{dupa}
\end{align} 
Observe also that
$$ Q_j(0, \ldots, 0, X_{1, j} - a_{1, j}, 0, \ldots, 0) = 0,$$
since $\sum_{i=1}^N X_{1, i}\geq G$ and
$X_{1, j} - a_{1, j} < G$, so Player $j$ will not be able to increase his fortune and, as a consequence to reach his goal $G$.

To complete the proof we will show that $Q_j$ is excessive. Then \cite{MS}*{Theorem 3.3.10} implies that a bold strategy is optimal for Player $j$ if all remaining players play boldly.
Therefore, we need to prove that if at the first stage of the game Player $j$ bets an amount $a_{1,j}<X_{1,j}$, i.e. less than his entire fortune, and then plays boldly for the rest of the game (if the game lasts till second round), then the expected return for him is not greater than his expected return would be if he played boldly at the first stage as well.  
Denote the expected return for Player $j$ who adopts this strategy by $\sigma_j(X_{1, 1}, \ldots, X_{1, j}, \ldots, X_{1, N})$.
Staking an amount $1\leq a_{1,j}\leq X_{1,j}$ means for him that 

\begin{align*}
    \sigma_j(X&_{1, 1}, \ldots, X_{1, j}, \ldots, X_{1, N}) \\ &=
    \Phi(X_{1, 1}; X_{1, 1}, \ldots, a_{1, j}, \ldots, X_{1, N})\cdot Q_j(  \sum_{i\neq j} X_{1, i}+ a_{1, j}, 0, \ldots, X_{1, j} - a_{1, j}, \ldots, 0) \\ &+ \ldots + \Phi(a_{1, j}; X_{1, 1}, \ldots, a_{1, j}, \ldots, X_{1, N})\cdot Q_j(0, \ldots, \sum_{i=1}^N X_{1, i}, \ldots, 0) + \ldots + \\&+ \Phi(X_{1, N}; X_{1, 1}, \ldots, a_{1, j}, \ldots, X_{1, N})\cdot Q_j(0, \ldots, X_{1, j} - a_{1, j}, \ldots, 0, \sum_{i\neq j} X_{1, i}+a_{1, j} ) \\&+ \bigg[1- \sum_{i\neq j} \Phi(X_{1, i}; X_{1, 1}, \ldots, a_{1, j}, \ldots, X_{1, N})  - \Phi(a_{1, j}; X_{1, 1}, \ldots, a_{1, j}, \ldots, X_{1, N})  \bigg]\\&\quad \cdot Q_j(0, \ldots, X_{1, j} - a_{1, j}, \ldots, 0).
    \end{align*}
In the above calculations, we have counted all the possibilities of winning by each player different from $j$ (in this case Player $j$ continues to play with fortune $X_{1, j} - a_{1, j}$), the case when Player $j$ wins (which is represented by the middle element), as well as casino rakes in the stake (the last term).	
		
To proceed we need to focus on the case where only two players remain in the game. By \eqref{ass} and using \eqref{dupa}, we have
\begin{align*}
    Q_j( a_{1, j} + \sum_{i\neq j} X_{1, i}, 0,& \ldots, X_{1, j} - a_{1, j}, \ldots, 0) \\&= \ldots =  Q_j(0, \ldots, X_{1, j} - a_{1, j}, \ldots, 0,  \sum_{i\neq j} X_{1, i} +a_{1, j}) 
		\\&= \Phi(X_{1, j}-a_{1,j}; \sum_{i\neq j} X_{1, i}+a_{1, j}, 0 ,\ldots, 0, X_{1, j} - a_{1, j}, 0 \ldots, 0).
    \end{align*}
\noindent
Finally, joining the above estimates, we arrive at
\begin{align*}
    \sigma(X_{1, 1}, &\ldots, X_{1, j}, \ldots, X_{1, N}) =
    \Phi(a_{1, j}; X_{1, 1}, \ldots, a_{1, j}, \ldots, X_{1, N})\cdot 1 \\&+ \sum_{i \neq j} \Phi(X_{1, i}; X_{1, 1}, \ldots, a_{1, j}, \ldots, X_{1, N})\cdot \Phi(X_{1, j}-a_{1,j}; \sum_{i\neq j} X_{1, i}+a_{1, j}, 0 ,\ldots, 0, X_{1, j} - a_{1, j}, 0 \ldots, 0).
\end{align*}

To show that $Q_j$ is excessive we need to verify the inequality
$$ Q_j(X_{1, 1}, \ldots, X_{1, N}) \geq \sigma_j(X_{1, 1}, \ldots, X_{1, N}).
$$
Note that it is equivalent to
\begin{align*}
\Phi(X_{1, j}; &X_{1, 1}, \ldots, X_{1, N}) \geq \Phi(a_{1, j}; X_{1, 1}, \ldots, a_{1, j}, \ldots, X_{1, N}) \\&+ \sum_{i \neq j} \Phi(X_{1, i}; X_{1, 1}, \ldots, a_{1, j}, \ldots, X_{1, N})\cdot \Phi(X_{1, j}-a_{1,j}; \sum_{i\neq j} X_{1, i}+a_{1, j}, 0 ,\ldots, 0, X_{1, j} - a_{1, j}, 0 \ldots, 0).
\end{align*}
Now, introduce functions $f, g$ as in the statement of the theorem, i.e.
$$ f(x_{1, j}) = \Phi(x_{1, j}; X_{1, 1}, \ldots, X_{1, j-1}, x_{1, j}, X_{1, j+1}, \ldots, X_{1, N})$$
and
$$g(a_{1, j}) = \sum_{i \neq j} \Phi(X_{1, i}; X_{1, 1}, \ldots, X_{1, j-1}, a_{1, j}, X_{1, j+1}, \ldots, X_{1, N}).$$
The inequality in question takes the form
$$ f(x_{1, j})- f(a_{1, j}) \geq g(a_{1, j}) \cdot f(x_{1, j} - a_{1, j}). $$
Note also that $f(0)=0$ by \eqref{00}.
Therefore, we reduced the problem to the inequality \eqref{main} and now the proof is completed. 
\qed

\section{Functional inequality}

We will begin this section with a fundamental observation about solutions to inequality \eqref{main}.

\begin{prop}\label{p1}
Assume that we are given two functions $f, g \colon S \to [0,1]$ and $f$ is positive on $S$. Then $f, g$ satisfy inequality \eqref{main} for all $a, x \in S$ such that $a \leq x$  if and only if
\begin{equation}
g(y) \leq \min \left\{ \frac{f(x) - f(y)}{f(x-y)} : x \in \{ y+1, \ldots , M \} \right\}, \quad y \in S.
\label{min}
\end{equation}
\end{prop}
\begin{proof}
Straightforward.
\end{proof}

From the above proposition, we have an immediate corollary.

\begin{cor}
Assume that $f \colon S \to (0,1]$ is a non-decreasing function and $g \colon S \to [0,1]$ is defined by the formula
\begin{equation}
g(y) = \min \left\{ \frac{f(x) - f(y)}{f(x-y)} : x \in \{ y+1, \ldots , M \} \right\}, \quad y \in S.
\label{g_min}
\end{equation}
Then $f, g$ satisfy inequality \eqref{main} for all $a, x \in S$ such that $a \leq x$.
\end{cor}

One can ask about solutions of a corresponding functional equation, when one replace inequality sign in \eqref{main} by equality, i.e.
\begin{equation}
f(x) - f(a) = g(a)f(x-a),
\label{eq}
\end{equation}
for all $a, x \in S$ such that $a \leq x$. However, it is not difficult to  find all solutions of  \eqref{eq}. Let us note the following observation.

\begin{prop}\label{pf}
Assume that we are given two functions $f, g \colon S \to \R$. Then $f, g$ satisfy functional equation \eqref{eq} for all $a, x \in S$ such that $a \leq x$  if and only if:
\begin{itemize}
	\item[(i)] $g=0$ and $f$ is constant on $S$, or
	\item[(ii)]   $f=0$ and $g$ is arbitrary on $S$, or
	\item[(iii)] $f(x) = f(1)x$ for all $x \in S$ and $g=1$ on $S$, or
	\item[(iv)] $f(x) = \alpha (g(0)^x-1)$ and $g(x) = g(1)^x$ for all $x \in S$ with some constant $\alpha \in \R$.
\end{itemize}
\end{prop}
\begin{proof}
The ''if'' implication is straightforward, thus we will justify the ''only if'' implication. Apply \eqref{eq} with $x=a$ to deduce that $g(a)f(0) = 0$ for all $a \in S$. Thus, either $g=0$ on $S$, or $f(0)=0$. We will discuss the second case. Put $a=0$ in \eqref{eq} to get that either $f=0$ on $S$, or $g(0)=1$. So, assume that $g(0)=1$. Next, observe that if for some $b \in S$, $b>0$ we have $f = 0$ on $\{1, 2, \dots , b\}$, then by \eqref{eq} we obtain $f(b+1)-f(1) = g(1)f(b)=0$, so $f(b+1)=0$. Therefore, we can assume that $f(1)\neq 0$. Without loss of generality assume that $f(1)=1$, by multiplying equation \eqref{eq} by a constant and replacing $f$ by $f(1)^{-1}f$, if necessary. Denote $c:=g(1)$ and use \eqref{eq} to get
$$f(x+1) - 1 = c f(x), \quad x \in S, \, x+1 \in S.$$
A straightforward induction leads to
$$f(y) = c^{y-1}+\dots +1, \quad y \in S, \, y>0.$$
Thus, if $c=1$, then $f(y) = y$ and $f(y) =  \frac{c^y-1}{c-1}$ if $c \neq 1$, which corresponds to cases (iii) and (iv), respectively. 
\end{proof}

\begin{rem}
The first three cases of Proposition \ref{pf} correspond to rather uninteresting possibilities in the game. First (when $g=0$ and $f$ is constant on $S$) is when the casino never wins and each player wins with the probability of $1/N$ regardless of the bets. The second (when $f=0$) is when the casino collects all the bets with probability one. The third option spoken of in Proposition \ref{pf} is not applicable, since if $f(1)\neq 0$, then condition \eqref{1}, which guarantees that the sum of probabilities of winning does not exceed one, is not satisfied. The fourth case, is on the contrary of definite interest since it corresponds to the power functions, which were considered as win probability functions in \cite{Ch1}, see also \cite{P2}*{Example 3.2}. 
\end{rem}

\section{Examples and final remarks}

A calculation on \cite{P2}*{p. 552, before Remarks 3.1} leads us to the following corollary.

\begin{cor}
Assume that $\f\colon [0,1] \to [0,1]$ is a super-multiplicative and super-additive function and $x_1, \dots , x_N \in S$ are fixed. Then function $\Phi\colon S\times P \to [0, 1]$ given by
$$\Phi(a_j;a_1, \dots a_N ) := \f\(\frac{a_j}{ \sum_{j=1}^Na_j } \)$$
satisfy assumptions of Theorem \ref{t1}.
\end{cor}

Two particular examples of function $\f$ which fulfils assumptions of the above corollary are given in  \cite{P2}*{Examples 3.1 and 3.2} as $\f(s) = w s$ with some $w \leq 1$ and $f(s) = s^p$ with $p \geq 1$.

We will terminate the paper with a few simple examples of probability functions that satisfy \eqref{fu} (except for the last one, which describes a situation when assumptions of our model are not satisfied). Then, we will compute corresponding functions $f$ and $g$ and we check whether inequality \eqref{main} is satisfied or not. This will show that our approach is more general than the one presented in the literature on red-and-black gambling. 

\begin{ex}
Assume that function $\Phi$ is defined as
$$ \Phi(a_j; a_1, \ldots , a_N):=\frac{1}{M\cdot N}[a_1+\dots + a_N-a_j]$$
when $a_j>0$ and $\Phi(0; a_1,\dots a_N )=0$. Factor $1/M\cdot N$ guarantees that $\Phi$ is a probability function and condition \eqref{1} holds. Moreover \eqref{fu} is satisfied by $\Phi$. It is easy to check that $f$ and $g$ are constant functions. Thus inequality \eqref{main} is satisfied if and only if $f=0$ or $g=0$.
This corresponds to uninteresting cases of the game (which should be of no surprise, since in this example the chance of winning of Player $j$ does not depend on his bet).
\end{ex}


\begin{ex}
Assume that function $\Phi$ is defined as
$$ \Phi(a_j; a_1, \ldots , a_N):=\frac{a_j}{M}$$
when $a_j>0$ and $\Phi(0; a_1,\dots a_N )=0$. Again, \eqref{fu} is satisfied by $\Phi$, $f$ and $g$ are also easy to find and inequality \eqref{main} is always satisfied. If (and only if) all the players will choose the bold strategy, i.e. $a_j=X_j$, then the sum of all $\Phi$'s equals $1$, which means that the casino has no chance to rack in their bets. Therefore, in this example the bold strategy is dominant.
\end{ex}

\medskip 

Further examples are easy and show that in the context of non-constant sum games with bet-dependent probabilities it is no longer relevant whether the game is sub-fair or super-fair. Namely, it is possible to construct a super-fair game in which a player should adopt a bold strategy and a sub-fair game with a timid strategy being optimal.

\begin{ex}
Assume that for every $j \in\{1, \ldots,  N\}$ and every choice of bets $(a_1, \ldots , a_N)$ we have
$$ \Phi(a_j; a_1, \ldots , a_N):=\frac{1}{N}.$$
In this situation, a timid strategy is always dominant for each player. 

One can modify this example by allowing the casino a positive probability of collecting all the bets, for example as follows:
$$ \Phi(a_j; a_1, \ldots , a_N):=\frac{1}{2N}.$$
Now the situation changes and with the probability of $1/2$ all the bets are taken by the casino. Still, it is unwise for a single player to deviate from a dominant timid strategy and increase his bet. Consequently, with timid strategies of all players the probability that any of them will win decreases as the initial sum is large.

Note however that if the players were allowed to cooperate, then they would adopt bold strategies to increase their expected payoffs. 
\end{ex}

\begin{ex}
Assume that the win probability function for the first player is equal to 
$$ \Phi(a_1, \ldots , a_N) := \left\{  \begin{array}{ll}
	0,  & \textrm{if } a_1 < x_1^0, \\
	1,  & \textrm{if } a_1 \geq x_1^0.
	\end{array}            
\right. $$
Thus, without any additional restrictions upon the remaining probability functions, we see that the bold strategy is a dominant strategy guaranteeing her a sure victory. Here equilibrium is not unique (provided the game lasts more than one turn), since any bet of Player 1 greater than or equal to $x_1^0$ is a dominant strategy. Note that in this example assumptions of our model are violated. Therefore, our conditions are by no means necessary for the bold profile.
\end{ex}

\medskip

\noindent \textbf{Funding.} 
The authors declare that no funds, grants, or other support were received during the preparation of this manuscript.

\medskip

\noindent \textbf{Competing Interests.} 
The authors have no relevant financial or non-financial interests to disclose.

\medskip

\noindent \textbf{Author Contributions.} 
Both authors contributed to the study equally.
Both authors read and approved the final manuscript.

\medskip

\noindent \textbf{Data availability statement.} Not applicable.

%

\end{document}